\documentclass[11pt,reqno]{amsart}
\usepackage{amsaddr}

\usepackage{a4wide}
\usepackage{amsthm}
\usepackage{amsmath}
\usepackage{amssymb}
\usepackage[all,cmtip]{xy}
\usepackage{color}
\usepackage[T1]{fontenc}
\usepackage{enumerate}

\usepackage{hyperref}

\theoremstyle{plain}
\newtheorem{thm}{Theorem}

\newtheorem{prop}[thm]{Proposition}
\newtheorem{cor}[thm]{Corollary}
\theoremstyle{definition}
\newtheorem{defn}[thm]{Definition}

\newtheorem{rem}[thm]{Remark}

\DeclareMathOperator{\identity}{id}

\DeclareMathOperator{\id}{id}

\newcommand{\C}{\mathbb{C}}

\begin{document}

\title[Non-associative Ore Extensions]{Non-associative Ore Extensions}

\author{Patrik Nystedt}
\address{University West,
Department of Engineering Science, 
SE-46186 Trollh\"{a}ttan, Sweden}

\author{Johan \"{O}inert}
\address{Blekinge Institute of Technology,
Department of Mathematics and Natural Sciences,
SE-37179 Karlskrona, Sweden}

\author{Johan Richter}
\address{M\"{a}lardalen University,
Academy of Education, Culture and Communication, \\
Box 883, SE-72123 V\"{a}ster\aa s, Sweden}

\email{patrik.nystedt@hv.se; johan.oinert@bth.se; johan.richter@mdh.se}


\begin{abstract}
We introduce non-associative Ore extensions, $S = R[X ; \sigma , \delta]$,
for any non-associa\-tive unital ring $R$ and any additive maps 
$\sigma,\delta : R \rightarrow R$ satisfying $\sigma(1)=1$ and $\delta(1)=0$.
In the special case when $\delta$ is
either left or right $R_{\delta}$-linear,
where $R_{\delta} = \ker(\delta)$, and $R$ is $\delta$-simple, i.e.
$\{ 0 \}$ and $R$ are the only $\delta$-invariant ideals of $R$,
we determine the ideal structure of the non-associative differential polynomial ring
$D = R[X ; \id_R , \delta]$. Namely, in that case, we show that
all ideals of $D$ are generated by monic polynomials in the 
center $Z(D)$ of $D$. We also show that $Z(D) = R_{\delta}[p]$ for 
a monic $p \in R_{\delta}[X]$, unique up to addition
of elements from $Z(R)_{\delta}$.
Thereby, we generalize classical results by Amitsur on
differential polynomial rings defined by derivations on associative and simple rings.
Furthermore, we use the ideal structure of $D$ to show that $D$ is simple if and only 
if $R$ is $\delta$-simple and $Z(D)$ equals the field $R_{\delta} \cap Z(R)$.
This provides us with a non-associative generalization of a result
by \"{O}inert, Richter, and Silvestrov.
This result is in turn used to show a non-associative version
of a classical result by Jordan concer\-ning simplicity
of $D$ in the cases when the characteristic of 
the field $R_{\delta} \cap Z(R)$ is either zero 
or a prime.
We use our findings to show simplicity results for
both non-associative versions of Weyl algebras
and non-associative differential polynomial rings 
defined by monoid/group actions on compact Hausdorff spaces.
\end{abstract}

\maketitle

\pagestyle{headings}

\section{Introduction}

In 1933 Ore \cite{ore1933} introduced a version
of non-commutative polynomial rings,
nowadays called Ore extensions, that have
become one of the most useful constructions
in ring theory. The Ore extensions
play an important role when
investigating cyclic algebras, enveloping rings of solvable Lie algebras,
and various types of graded rings 
such as group rings and crossed products,
see e.g. \cite{cohn1977}, \cite{jacobson1999},
\cite{mcconell1988} and \cite{rowen1988}.
They are also a natural source of
examples and counter-examples in ring theory,
see e.g. \cite{bergman1964} and \cite{cohn1961}.
Furthermore, various special cases
of Ore extensions are used as
tools in diverse analytical settings,
such as differential-, pseudo-differential and
fractional differential operator rings
\cite{goodearl1983} and 
$q$-Heisenberg algebras \cite{hellstromsilvestrov2000}.

Let us recall the definition of an (associative) Ore extension.
Let $S$ be a unital ring.
Take $x \in S$ and let $R$ be a subring of $S$ 
containing $1$, the multiplicative identity element of $S$.
\begin{defn}\label{defore}
The pair $(S,x)$ is called an {\it Ore extension} of $R$ 
if the following axioms hold:
\begin{itemize}
\item[(O1)] $S$ is a free left $R$-module with basis
$\{ 1,x,x^2,\ldots \}$;
\item[(O2)] $xR \subseteq R + Rx$;
\item[(O3)] $S$ is associative.
\end{itemize}
If (O2) is replaced by 
\begin{itemize}
\item[(O2)$'$] $[x,R] \subseteq R$; 
\end{itemize}
then $(S,x)$ is called a {\it differential polynomial ring over} $R$. 
\end{defn}

Recall that $[x,R]$ denotes the set of finite sums of elements 
of the form $[x,r]=xr-rx$, for $r \in R$.

To construct Ore extensions, one considers
generalized polynomial rings $R[X ; \sigma , \delta]$ over an associative ring $R$, 
where $\sigma$ is a ring endomorphism of $R$,
respecting 1, and $\delta$ is a $\sigma$-derivation of $R$, i.e. 
an additive map $R \rightarrow R$ satisfying
$\delta(ab) = \sigma(a) \delta(b) + \delta(a) b$, for $a,b \in R$. 
Let $\mathbb{N}$ denote the set of non-negative integers. As an additive group $R[X ; \sigma , \delta]$  is equal to the usual polynomial ring $R[X].$ 
The ring structure on $R[X ; \sigma , \delta]$ 
is defined on monomials by
\begin{equation}\label{productmonomials}
a X^m \cdot b X^n = \sum_{i \in \mathbb{N}} a \pi_i^m(b) X^{i+n},
\end{equation}
for $a , b \in R$ and $m,n \in \mathbb{N}$, 
where $\pi_i^m$ denotes the sum of all the
${m \choose i}$ possible compositions of $i$
copies of $\sigma$ and $m-i$ copies of $\delta$ in arbitrary order
(see equation (11) in \cite{ore1933}).
Here we make the convention that
$\pi_i^m(b) = 0$, for $i,m \in \mathbb{N}$
such that $i > m$.
The product \eqref{productmonomials} makes the pair
$(R[X ; \sigma , \delta],X)$ an Ore extension of $R$.
In fact, (O1) and (O2) are immediate and (O3)
can be shown in several different ways, see e.g. 
\cite[Proposition 7.1]{bergman1978},
\cite{nystedt2013},
\cite{richter2014},
\cite[Proposition 1.6.15]{rowen1988},
or Proposition \ref{newproof} in the present article for yet another proof.
This class of generalized polynomial rings provides us with 
all Ore extensions of $R$ .
Indeed, given an Ore extension $(S,x)$ of $R$,
then define the maps $\sigma : R \rightarrow R$ and 
$\delta : R \rightarrow R$ by the relations 
$x a = \delta(a) + \sigma(a) x$, for $a \in R$.
Then it follows that 
$\sigma$ is a ring endomorphism of $R$,
respecting 1, $\delta$ is a $\sigma$-derivation of $R$
and there is a unique well defined 
ring isomorphism $f : S \rightarrow R[X;\sigma,\delta]$ subject to 
the relations $f(x) = X$ and $f|_R = \id_R$.
If $(S,x)$ is a differential polynomial ring over $R$,
then $\sigma = \id_R$ and $\delta$ is a derivation on $R$.

Many different properties of associative Ore extensions, such as
when they are integral domains, principal domains,
prime or noetherian have been studied by numerous authors
(see e.g. \cite{cozzens1975} or \cite{mcconell1988} for surveys).
Here we focus on the property of simplicity of 
differential polynomial rings $D = R[X ; \id_R , \delta]$.
Recall that $\delta$ is called {\it inner}
if there is $a \in R$ such that
$\delta(r) = ar - ra$, for $r \in R$.
In that case we write $\delta = \delta_a$.
If $\delta$ is not inner, then $\delta$ is called {\it outer}. 
We let the {\it characteristic} of a
ring $R$ be denoted by ${\rm char}(R)$. 
In an early article by Jacobson \cite{jacobson1937}
it is shown that if $\delta$ is outer and $R$ is a division ring
with ${\rm char}(R)=0$, then $D$ is simple.
The case of positive characteristic is more complicated 
and $D$ may contain non-trivial ideals.
In fact, Amitsur \cite{amitsur1950} has shown that 
if $R$ is a division ring with ${\rm char}(R)= p > 0$, 
then every ideal of $D$ is generated by a polynomial, all of whose
monomials have degrees which are multiples of $p$.
A few years later Amitsur \cite{amitsur1957} generalized this 
result to the case of simple $R$. To describe this generalization
we need to introduce some more notation.
Let $T$ be a subring of $D$.
Let $Z(T)$ denote the {\it center} of $T$, i.e. the set of all elements in $T$
that commute with every element of $T$.
If $T$ is a subring of $R$, then put $T_{\delta} = T \cap \ker(\delta)$.
Note that if $R$ is simple, then $Z(R)$ is a field.
Therefore, in that case, ${\rm char}(R) = {\rm char}(Z(R))$ and hence
${\rm char}(R)$ is either zero or a prime $p > 0$.

\begin{thm}[Amitsur \cite{amitsur1957}]\label{amitsurtheorem}
Suppose that $R$ is a simple associative ring and let
$\delta$ be a derivation on $R$.
If we put $D = R[X ; \id_R , \delta]$, then the following assertions hold:
\begin{itemize}
\item[(a)] Every ideal of $D$ is generated by a unique monic polynomial in $Z(D)$;
\item[(b)] There is a monic $b \in R_{\delta}[X]$, unique up to addition of 
elements from $Z(R)_{\delta}$, such that
$Z(D) = Z(R)_{\delta}[b]$;
\item[(c)] If ${\rm char}(R)=0$ and $b \neq 1$, then there is $c \in R_{\delta}$
such that $b = c + X$. In that case, $\delta = \delta_c$;
\item[(d)] If ${\rm char}(R) = p > 0$ and $b \neq 1$, then there is $c \in R_{\delta}$
and $b_0,\ldots,b_n \in Z(R)_{\delta}$, with $b_n=1$, such that
$b = c + \sum_{i=0}^n b_i X^{p^i}$.
In that case, $\sum_{i=0}^n b_i \delta^{p^i} = \delta_c$.
\end{itemize}
\end{thm}

The condition that $R$ is simple in the above theorem is not
necessary for simplicity of $D = R[X ; \id_R , \delta]$.
Consider e.g. the well known example of the first Weyl algebra
where $R = K[Y]$, $K$ is a field with ${\rm char}(K)=0$ and 
$\delta$ is the usual derivative on $R$
(for more details, see e.g. \cite[Example 1.6.32]{rowen1988}).
However, $\delta$-simplicity of $R$ is always a necessary
condition for simplicity of $D$ 
(see \cite[Lemma 4.1.3(i)]{jordan1975} or Proposition \ref{sigmadeltasimple}).
Recall that an ideal $I$ of $R$ is called $\delta$-invariant
if $\delta(I) \subseteq I$.
The ring $R$ is called $\delta$-simple if $\{ 0 \}$ and $R$
are the only $\delta$-invariant ideals of $R$.
Note that if $R$ is $\delta$-simple, then
the ring $Z(R)_{\delta}$ is always a field.
Therefore, in that case, ${\rm char}(R) = {\rm char}(Z(R)_{\delta})$ and hence
${\rm char}(R)$ is either zero or a prime $p > 0$.
Jordan \cite{jordan1975} (and Cozzens and Faith \cite{cozzens1975}
in a special case) has shown the following result.

\begin{thm}[Jordan \cite{jordan1975}]\label{jordantheorem}
Suppose that $R$ is a $\delta$-simple associative ring and let
$\delta$ be a derivation on $R$.
If we put $D = R[X ; \id_R , \delta]$, then the following assertions hold:
\begin{itemize}
\item[(a)] If ${\rm char}(R)=0$, then $D$ is simple if and only if
$\delta$ is outer;
\item[(b)] If ${\rm char}(R) = p > 0$, then $D$ is simple if and only if
no derivation of the form $\sum_{i=0}^n b_i \delta^{p^i}$,
$b_i \in Z(R)_{\delta}$, and $b_n=1$, is an inner derivation 
induced by an element in $R_{\delta}$.
\end{itemize}
\end{thm} 

In the case when $R$ is commutative, Cozzens and Faith \cite{cozzens1975}
(for integral domains $R$ of prime characteristic)
and Goodearl and Warfield \cite{goodearl1982} (in the general case) have shown that
$R[x ; \id_R, \delta]$ is simple 
if and only if $R$ is $\delta$-simple and $R$ is 
infinite-dimensional as a vector space over $R_{\delta}$.
If one has a family of commuting derivations,
then one can form a differential polynomial ring in several variables. 
The articles \cite{malm1988}, \cite{posner1960} and  \cite{voskoglou1985}
consider the question when such rings are simple. In the preprint \cite{nystedtoinertrichter} the authors of the present article study when non-associative differential polynomial rings in several variables are simple.

In the simplicity results mentioned above, a distinction
is often made between the cases when the characteristic of $R$
is zero or the characteristic of $R$ is prime.
Special attention is also often paid to the case when $R$ is commutative.
However, in \cite{oinert2013} \"{O}inert, Richter and Silvestrov have shown
the following simplicity result that holds
for all associative differential polynomial rings regardless of characteristic.

\begin{thm}[\"{O}inert, Richter and Silvestrov \cite{oinert2013}]\label{richtersilvestrovoinert}
If $R$ is associative and $\delta : R \rightarrow R$
is a derivation, then 
$D = R[X ; \id_R , \delta]$ is simple 
if and only if $R$ is $\delta$-simple and $Z(D)$ is a field.
\end{thm}

In this article, we address the question of what it should
mean for a pair $(S,x)$ to be a non-associative Ore extension of $R$ 
and when the resulting rings are simple.
It seems to the authors of the present article
that this question has not previously been analysed in the literature.
Let us briefly describe the train of reasoning that lead the authors to
their definition of such objects.
The product \eqref{productmonomials} 
equips the set $R[X ; \sigma , \delta]$ of generalized polynomials over 
any non-associative ring $R$ 
with a well defined non-associative ring structure 
for any additive maps 
$\sigma : R \rightarrow R$ and $\delta : R \rightarrow R$
satisfying $\sigma(1)=1$ and $\delta(1)=0$.
We wish to adapt the axioms (O1), (O2) and (O3) 
to the non-associative situation so that
the resulting collection of non-associative rings
coincides with this family of generalized polynomial rings.
It turns out that this 
happens precisely when $x$
belongs to the right and middle nucleus of $S$.
To be more precise, let $S$ be a non-associative ring, by this we mean that $S$ is an additive abelian group equipped with  a multiplication
which is distributive with respect to addition and which has multiplicative identity $1$.
We suggest the following. 
\begin{defn}\label{defnonore}
The pair $(S,x)$ is called a {\it non-associative Ore extension} of $R$ if
the following axioms hold:
\begin{itemize}
\item[(N1)] $S$ is a free left $R$-module with basis
$\{ 1,x,x^2,\ldots \}$;
\item[(N2)] $xR \subseteq R + Rx$;
\item[(N3)] $(S,S,x) = (S,x,S) = \{ 0 \}$.
\end{itemize}
If (N2) is replaced by 
\begin{itemize}
\item[(N2)$'$] $[x,R] \subseteq R$; 
\end{itemize}
then $(S,x)$ is called a {\it non-associative differential polynomial ring over} $R$. 
\end{defn}
For non-empty subsets
$A$, $B$ and $C$ of $S$,
we let $(A,B,C)$ denote the set of finite sums of elements
of the form $(a,b,c) = (ab)c - a(bc)$, for $a \in A$,
$b \in B$ and $c \in C$.
Note that from (N3) it follows that the 
element $x$ is power associative, so that the 
symbols $x^i$, for $i \in \mathbb{N}$, are well defined.

Here is an outline of this article.

In Section \ref{nonassociativeringtheory},
we gather some well known facts from non-associative ring
and module theory that we need in the sequel.
In particular, we state our conventions concerning
modules over non-associative rings and
what a basis
should mean
in that situation.

In Section \ref{oreextensions}, we show that
there is a bijection between the set of non-associative Ore extensions of $R$
and the set of generalized polynomial rings $R[X ; \sigma , \delta]$ over $R$, 
where $\sigma$ and $\delta$ are additive maps $R \rightarrow R$
such that $\sigma(1)=1$ and $\delta(1)=0$.
If $T$ is a subset of $R$, then we put
$T_{\delta}^{\sigma} = \{ a \in T \mid \sigma(a)=a \ \mbox{and} \ \delta(a)=0 \}$,
$T_{\delta} = T_{\delta}^{\id_R}$ and $T^{\sigma} = T_{0}^{\sigma}$. 
In Section \ref{oreextensions}, we introduce
the class of {\it strong} non-associative Ore extensions
(see Definition \ref{defstrong}).
These correspond to generalized polynomial rings
$R[X ; \sigma , \delta]$,
where $\sigma$ is a, what we call, {\it fixed point homomorphism} of $R$
and $\delta$ is a, what we call, $\sigma$-{\it kernel derivation} of $R$.
By this we mean that
$\sigma$ and $\delta$ are maps $R \rightarrow R$ satisfying
$\sigma(1)=1$, $\delta(1)=0$ and both of them are right $R_{\delta}^{\sigma}$-linear
or both of them are left $R_{\delta}^{\sigma}$-linear.
Clearly, every classical derivation is a $\sigma$-kernel derivation with $\sigma=\id_R$ and every classical homomorphism is a fixed point homomorphism. In general, a $\sigma$-kernel derivation with $\sigma=\id_R$  will simply be called a \emph{kernel derivation}.  

In Section \ref{simplicity}, we introduce $\sigma$-$\delta$-simplicity
for rings $R$, where $\sigma$ and $\delta$ are additive maps $R \rightarrow R$
such that $\sigma(1)=1$ and $\delta(1)=0$ (see Definition \ref{defsigmadeltasimple}). 
We show that $\sigma$-$\delta$-simplicity of $R$ is a necessary
condition for simplicity of non-associative Ore extensions 
$R[X ; \sigma , \delta]$ (see Proposition \ref{sigmadeltasimple}).
We also show that if $R$ is $\sigma$-$\delta$-simple,
then $Z(R)_{\delta}^{\sigma}$
is a field (see Proposition \ref{R-sigmadelta-simple-Z-field}).
Thus, in that case, we get that ${\rm char}(R) = {\rm char}( Z(R)_{\delta}^{\sigma} )$
and hence that ${\rm char}(R)$ is either zero or a prime $p > 0$.
In Section \ref{simplicity}, we prove the following
non-associative generalization of Theorems \ref{amitsurtheorem},
\ref{jordantheorem} and \ref{richtersilvestrovoinert}.

\begin{thm}\label{maintheorem}
Suppose that $R$ is a non-associative ring and that
$\delta$ is a kernel derivation on $R$.
If we put $D = R[X ; \id_R , \delta]$, then the following assertions hold:
\begin{itemize}
\item[(a)] If $R$ is $\delta$-simple, then 
every ideal of $D$ is generated by a unique monic polynomial in $Z(D)$;
\item[(b)] If $R$ is $\delta$-simple,
then there is a monic $b \in R_{\delta}[X]$, unique up to addition 
of elements from $Z(R)_{\delta}$, such that $Z(D) = Z(R)_{\delta}[b]$; 
\item[(c)] $D$ is simple if and only if
$R$ is $\delta$-simple and $Z(D)$ is a field.
In that case $Z(D) = Z(R)_{\delta}$
in which case $b=1$;
\item[(d)] If $R$ is $\delta$-simple,
$\delta$ is a derivation on $R$ and
${\rm char}(R)=0$,
then either $b=1$ or there is $c \in R_{\delta}$ such that $b = c + X$.
In the latter case, $\delta = \delta_c$;
\item[(e)] If $R$ is $\delta$-simple,
$\delta$ is a derivation on $R$ and
${\rm char}(R)=p>0$,
then either $b=1$ or there is $c \in R_{\delta}$ and 
$b_0,\ldots,b_n \in Z(R)_{\delta}$, with $b_n=1$,
such that $b = c + \sum_{i=0}^n b_i X^{p^i}$.
In the latter case,
$\sum_{i=0}^n b_i \delta^{p^i} = \delta_c$.
\end{itemize}
\end{thm}

In Section \ref{sectionweyl},
we introduce non-associative versions of 
the first Weyl algebra
(see Definition \ref{definitionweyl}) and we show that they are often 
simple regardless of the characteristic (see Theorem \ref{theoremweyl}).

In Section \ref{sectiondynamics}, 
we introduce a special class of $\sigma$-kernel derivations
induced by ring automorphisms
(see Definition \ref{definitionkernel}).
This yields simplicity results for a differential 
polynomial ring analogue of the quantum plane
(see Theorem \ref{theoremquantumtorus}) and
for differential polynomial rings defined by monoid/group actions
on compact Hausdorff spaces (see Theorem~\ref{NYtheoremdynamics} and Theorem~\ref{theoremdynamics}).

In Section \ref{sectionassociative},
we show that if the coefficients are associative, 
then we can often obtain simplicity of the differential polynomial ring
just from the assumption that the map $\delta$ is not a derivation.

\section{Preliminaries from Non-associative Ring Theory}\label{nonassociativeringtheory}

In this section, we recall some notions from non-associative
ring theory that we need in subsequent sections. 
Although the results stated in this
section are presumably rather well known, we have, for the convenience of the reader,
nevertheless chosen to include proofs of these statements.

Throughout this section, 
$R$ denotes a non-associative ring.
By this we mean that $R$ is an additive abelian group in which a multiplication
is defined, satisfying left and right distributivity.
We always assume that $R$ is unital and that the multiplicative identity 
of $R$ is denoted by $1$.
The term ''non-associative'' should be interpreted 
as ''not necessarily associative''.
Therefore all associative rings are non-associative.
If a ring is not associative,
we will use the term ''not associative ring''. 

By a {\it left module} over $R$ we mean an additive group $M$
equipped with a biadditive map 
$R \times M \ni (r,m) \mapsto rm \in M$.
In that case, we say that a subset $B$ of $M$ is a basis
if for every $m \in M$, there are unique $r_b \in R$, for $b \in B$,
such that $r_b = 0$ for all but finitely many $b \in B$,
and $m = \sum_{b \in B} r_b b$.
{\it Right modules} over $R$ and bases are defined in an analogous manner.

Recall that the \emph{commutator} $[\cdot,\cdot] : R \times R \rightarrow R$ 
and the \emph{associator} $(\cdot,\cdot,\cdot) : R \times R \times R \rightarrow R$ 
are defined by $[r,s]=rs-sr$ and
$(r,s,t) = (rs)t - r(st)$ for all $r,s,t \in R$, respectively.
The \emph{commuter} of $R$, denoted by $C(R)$,
is the subset of $R$ consisting
of elements $r \in R$ such that $[r,s]=0$
for all $s \in R$.
The \emph{left}, \emph{middle} and \emph{right nucleus} of $R$,
denoted by $N_l(R)$, $N_m(R)$ and $N_r(R)$, respectively, are defined by 
$N_l(R) = \{ r \in R \mid (r,s,t) = 0, \ \mbox{for} \ s,t \in R\}$,
$N_m(R) = \{ s \in R \mid (r,s,t) = 0, \ \mbox{for} \ r,t \in R\}$, and
$N_r(R) = \{ t \in R \mid (r,s,t) = 0, \ \mbox{for} \ r,s \in R\}$.
The \emph{nucleus} of $R$, denoted by $N(R)$,
is defined to be equal to $N_l(R) \cap N_m(R) \cap N_r(R)$.
From the so-called \emph{associator identity} 
$u(r,s,t) + (u,r,s)t + (u,rs,t) = (ur,s,t) + (u,r,st)$,
which holds for all $u,r,s,t \in R$, it follows that
all of the subsets $N_l(R)$, $N_m(R)$, $N_r(R)$ and $N(R)$
are associative subrings of $R$.
The \emph{center} of $R$, denoted by $Z(R)$, is defined to be equal to the 
intersection $N(R) \cap C(R)$.
It follows immediately that $Z(R)$ is an associative, unital
and commutative subring of $R$.

\begin{prop}\label{intersection}
The following three equalities hold:
\begin{align}
Z(R) &= C(R) \cap N_l(R) \cap N_m(R); \label{FIRST}\\
Z(R) &= C(R) \cap N_l(R) \cap N_r(R); \label{SECOND}\\
Z(R) &= C(R) \cap N_m(R) \cap N_r(R). \label{THIRD}
\end{align}
\end{prop}

\begin{proof}
We only show \eqref{FIRST}. The equalities \eqref{SECOND} and \eqref{THIRD} are shown 
in a similar way and are therefore left to the reader.
It is clear that $Z(R) \subseteq C(R) \cap N_l(R) \cap N_m(R)$.
Now we show the reversed inclusion. 
Take $r \in C(R) \cap N_l(R) \cap N_m(R)$. 
We need to show that $r \in N_r(R)$.
Take $s,t \in R$.
We wish to show that $(s,t,r)=0$, i.e. $(st)r = s(tr)$.
Using that $r\in C(R) \cap N_l(R) \cap N_m(R)$
we get $(st)r = r(st) = (rs)t = (sr)t = s(rt) = s(tr)$.
\end{proof}

\begin{prop}\label{centerInvClosed}
If $r \in Z(R)$ and $s \in R$ satisfy $rs = 1$,
then $s \in Z(R)$.
\end{prop}

\begin{proof}
Let $r\in Z(R)$ and suppose that $rs=1$.
First we show that $s \in C(R)$.
To this end, take $u \in R$.
Then $su = (su)1 = (su)(rs) = (r(su))s=
((rs)u)s = (1u)s = us$ and hence $s \in C(R)$.
By Proposition \ref{intersection}, we are done if we can show
$s \in N_l(R) \cap N_m(R)$. To this end, take $v \in R$.
Then $s(uv) = s((1u)v)= s(((rs) u)v) = 
(rs) ( (su) v ) = 1( (su)v ) = (su)v$
which shows that $s \in N_l(R)$. 
We also see that
$(us)v = (us)(1v) = (us) ( (rs) v ) =
( u (rs) ) (sv) = (u1)(sv) = u(sv)$
which shows that $s \in N_m(R)$.
\end{proof}

\begin{prop}\label{centerfield}
If $R$ is simple, then $Z(R)$ is a field.
\end{prop}

\begin{proof}
We already know that $Z(R)$ is a unital commutative ring.
What is left to show is that every non-zero element of $Z(R)$
has a multiplicative inverse in $Z(R)$.
To this end, take a non-zero $r \in Z(R)$.
Then $Rr$ is a non-zero ideal of $R$.
Since $R$ is simple, this implies that $R = Rr$.
In particular, we get that there is $s \in R$
such that $1 = sr$. By Proposition \ref{centerInvClosed},
we get that $s \in Z(R)$ and we are done.
\end{proof}

\section{Non-associative Ore extensions}\label{oreextensions}

In this section, we show that
there is a bijection between the set of (strong) non-associative Ore extensions of $R$
and the set of generalized polynomial rings $R[X ; \sigma , \delta]$ over $R$, 
where $\sigma$  (is a fixed point homomorphism) and 
$\delta$ (is a $\sigma$-kernel derivation) are
additive maps $R \rightarrow R$
such that $\sigma(1)=1$ and $\delta(1)=0$
(see Proposition \ref{sufficientpolynomial} and Proposition \ref{necessarypolynomial}).
We also show that if
$S = R[X ; \sigma , \delta]$ is a generalized polynomial ring, then 
$S$ is associative if and only if
$R$ is associative, $\sigma$ is a ring endomorphism and
$\delta$ is a $\sigma$-derivation
(see Proposition \ref{newproof}).

Throughout this section, $R$ denotes a non-associative ring.

\begin{defn}\label{definitionpolynomials}
By a formal set of polynomials $R[X]$ over $R$ we mean the collection 
of functions $f : \mathbb{N} \rightarrow R$ with the
property that $f(n)=0$ for all but finitely many $n \in \mathbb{N}$.
If $f,g \in R[X]$ and $r,s \in R$, then we define $rf + sg \in R[X]$
by the relation $(rf + sg)(n) = rf(n) + sg(n)$,
for $n \in \mathbb{N}$. 
If we for each $n \in \mathbb{N}$, let $X^n \in R[X]$ be defined
by $X^n(m) = 1$, if $m=n$, and $X^n(m)=0$, if $m \neq n$,
then $R[X]$ is a free left $R$-module with 
$B = \{ X^n \}_{n \in \mathbb{N} }$ as a basis.
In fact, for each $f \in R[X]$, we have that
$f = \sum_{n \in \mathbb{N}} f(n) X^n$.
By the degree of $f$, denoted by $\deg(f)$, we mean
the supremum of $\{ -\infty \} \cup \{ n \in \mathbb{N} \mid f(n) \neq 0 \}$.   
If $f \neq 0$, then we call $f( \deg(f) )$ the {\it leading coefficient of} $f$.
If the leading coefficient of $f$ is 1, then we say that $f$ is {\it monic}.
\end{defn}

\begin{defn}
Let $\sigma : R \rightarrow R$ and $\delta : R \rightarrow R$
be additive maps such that $\sigma(1)=1$ and $\delta(1)=0$. 
By the generalized polynomial ring
$R[X ; \sigma , \delta]$ over $R$ defined by $\sigma$ and $\delta$ 
we mean the set $R[X]$ of formal 
polynomials over $R$ equipped with the product defined 
on monomials by the relation \eqref{productmonomials}.
We will often identify each $r \in R$
with $rX^0$. It is clear that $R[X ; \sigma , \delta]$ is a 
non-associative ring with $1 = X^0$. 
It is also clear that $X$ is power associative
so that $X^n$, for $n > 0$, is in fact equal to the product of $X$ with itself $n$ times. 
\end{defn}

\begin{defn}\label{defstrong}
Suppose that $(S,x)$ is a non-associative Ore extension of $R$.
Put $R_x = \{ a \in R \mid ax = xa \}$.
We say that $(S,x)$ is {\it strong} if at least one of the following axioms holds:
\begin{itemize}
\item[(N4)] $(x,R,R_x) = \{ 0 \}$;
\item[(N5)] $(x,R_x,R) = \{ 0 \}$.
\end{itemize}
In that case we call $R_x$ {\it the ring of constants of $R$}.
If $(S,x)$ is a non-associative differential polynomial ring,
then we say that it is strong if it is strong as a 
non-associative Ore extension.
\end{defn}

The usage of the term ''ring'' in Definition \ref{defstrong}
is justified by the next result.

\begin{prop}
If $(S,x)$ is a strong non-associative Ore extensions of $R$, 
then $R_x$ is a subring of $R$.
\end{prop}

\begin{proof}
It is clear that $R_x$ is an additive subgroup of $R$ containing 1.
Now we show that $R_x$ is multiplicatively closed.
Take $a,b \in R_x$. Then
$(ab)x \stackrel{(N3)}{=} a(bx) \stackrel{[ b \in R_x ]}{=} 
a(xb) \stackrel{(N3)}{=} (ax)b \stackrel{[a \in R_x ]}{=}
(xa)b = x(ab)$. The last equality follows from the strongness of $(S,x)$.
Therefore $ab \in R_x$.
\end{proof}

\begin{prop}\label{sufficientpolynomial}
Every generalized polynomial ring $S = R[X ; \sigma , \delta]$ over $R$
(with $\sigma$ a fixed point homomorphism and
$\delta$ a $\sigma$-kernel derivation)
is a (strong) non-associa\-tive Ore extension of $R$ with $x = X$.
\end{prop}

\begin{proof}
We first show the ''non-strong'' statement.
From Definition \ref{definitionpolynomials}, we know that
$S$ is free as a left $R$-module with $B$ as a basis.
Therefore (N1) holds.
Also $Rx = RX^0 \cdot 1X = \delta(R) + \sigma(R)X =
\delta(R) + \sigma(R)x \subseteq R + Rx$.
Therefore (N2) holds.
Now we show (N3).
Suppose that $a,b \in R$ and $m,n \in \mathbb{N}$. 
Then we get that
$(aX^m \cdot bX^n) \cdot X = 
\sum_{i \in \mathbb{N}} a \pi_i^m(b) X^{i+n} \cdot X = 
\sum_{i \in \mathbb{N}} a \pi_i^m(b) X^{i+n+1} =a
X^m \cdot (bX^{n+1}) = aX^m \cdot (bX^n \cdot X).$
Next we get that
$(aX^m \cdot X) \cdot bX^n = 
aX^{m+1} \cdot bX^n =
\sum_{i \in \mathbb{N}} a \pi_i^{m+1}(b) X^{i+n} =
\sum_{i \in \mathbb{N}} a \pi_i^m(\delta(b)) X^{i+n} +
\sum_{i \in \mathbb{N}} a \pi_{i-1}^m(\sigma(b)) X^{i+n} =  
aX^m \cdot ( \delta(b) X^n + \sigma(b)X^{n+1} ) =
aX^m \cdot (X \cdot bX^n).$
Now we show the ''strong'' statement.
Note that $R_X = R_{\delta}^{\sigma}$.
Suppose first that both $\sigma$ and $\delta$
are right $R_{\delta}^{\sigma}$-linear.
We show (N4).
To this end, take $a \in R$ and $b \in R_X$. Then
$ (X \cdot a) \cdot b = 
(\delta(a) + \sigma(a)X) \cdot b 
\stackrel{(N3)}{=}
\delta(a) b + \sigma(a) (X b) = [b \in R_{\delta}^{\sigma} ]= 
\delta(a)b + \sigma(a) (bX) 
\stackrel{(N3)}{=}
\delta(a)b + (\sigma(a) b)X.$
Since $\sigma$ and $\delta$
are right $R_{\delta}^{\sigma}$-linear, we get that
$ (X \cdot a) \cdot b = \delta(ab) + \sigma(ab)X = X \cdot (ab)$.
Suppose now that both $\sigma$ and $\delta$
are left $R_{\delta}^{\sigma}$-linear.
We show (N5).
To this end, take $a \in R_X$ and $b \in R$. Then
$( X \cdot a ) \cdot b = [a \in R_{\delta}^{\sigma}] = (a \cdot X) \cdot b
\stackrel{(N3)}{=}
a \cdot (X b) = a \cdot (\delta(b) + \sigma(b)X) =
a \delta(b) + a \sigma(b)X.$
Since $\sigma$ and $\delta$
are left $R_{\delta}^{\sigma}$-linear, we get that
$( X \cdot a) \cdot b = \delta(ab) + \sigma(ab)X = X \cdot (ab).$
\end{proof}

\begin{prop}\label{necessarypolynomial}
Every non-associative Ore extension of $R$ is isomorphic
to a generalized polynomial ring $R[X ; \sigma, \delta]$.
If the non-associative Ore extension is strong, then $\sigma$ is a fixed point homomorphism and $\delta$ is a $\sigma$-kernel derivation. 
\end{prop}

\begin{proof}
We first show the ''non-strong'' statement.
Suppose that $S$ is a non-associative Ore extension of $R$
defined by the element $x \in S$. Take $a,b \in R$.
By (N1) and (N2), we get that $xa = \delta(a) + \sigma(a)x$,
for some unique $\delta(a),\sigma(a) \in R$.
Hence this defines functions $\sigma : R \rightarrow R$ and 
$\delta : R \rightarrow R$.
By distributivity of $S$, we get the relation
$x(a+b) = xa + xb$ which implies that 
$\sigma(a+b) = \sigma(a) + \sigma(b)$ and
$\delta(a + b) = \delta(a) + \delta(b)$.
From the relation $x1 = x$ we get that $\sigma(1)=1$ and $\delta(1)=0$.
Define $f : S \rightarrow R[X ; \sigma, \delta]$ by the additive extension
of the relations $f( a x^m ) = a X^m$, for $a \in R$ and $m \in \mathbb{N}$.
Then clearly $f$ is an isomorphism of additive groups.
What is left to show is that $f$ respects multiplication.
Take $a,b \in R$ and $m,n \in \mathbb{N}$.
We claim that $(a x^m)(b x^n) = \sum_{i \in \mathbb{N}} a \pi_i^m(b) x^{i+n}$.
If we assume that the claim holds, then
$f( (a x^m)(b x^n) ) = 
f( \sum_{i \in \mathbb{N}} a \pi_i^m(b) x^{i+n} ) = 
\sum_{i \in \mathbb{N}} a \pi_i^m(b) X^{i+n} =
(a X^m) \cdot (b X^n) = f( ax^n ) \cdot f(b x^m)$.
Now we prove the claim by induction over $m$.

First we show the base case $m=0$.
By (N2) we get that $x \in N_r(S)$.
Therefore $x^n \in N_r(S)$ and hence we get that
$(a x^0)(b x^n) = a (b x^n) = (a b)x^n = a \pi_0^0(b) x^n =
\sum_{i \in \mathbb{N}} a \pi_i^0(b) x^{i+n}$.

Next we show the induction step.
Suppose that the claim holds for some $m \in \mathbb{N}$. 
By (N2), we get that $x \in N_m(S) \cap N_r(S)$.
Therefore all powers of $x$ also belong to $N_m(S) \cap N_r(S)$
and hence we get that
$(a x^{m+1}) (b x^n) = (a( x^m x)) (b x^n) = 
((a x^m)x) (b x^n) = (a x^m) ( x (b x^n) ) 
= (a x^m) ( (xb) x^n ) = (a x^m) ( ( \delta(b) + \sigma(b) x ) x^n ) =
(a x^m) ( \delta(b) x^n + \sigma(b) x^{n+1} ) 
= (a x^m)( \delta(b) x^n ) + (a x^m)( \sigma(b) x^{n+1} ).$
By the induction hypothesis the last expression equals
$\sum_{i \in \mathbb{N}} a \pi_i^m( \delta(b) ) x^{i+n} + 
\sum_{i \in \mathbb{N}} a \pi_i^m( \sigma(b) ) x^{i+n+1} = 
\sum_{i \in \mathbb{N}} a \pi_i^m( \delta(b) ) x^{i+n} + 
\sum_{i \in \mathbb{N}} a \pi_{i-1}^m( \sigma(b) ) x^{i+n} = 
\sum_{i \in \mathbb{N}} a [ \pi_i^m( \delta(b) ) + \pi_{i-1}^m( \sigma(b) ) ] x^{i+n} =
\sum_{i \in \mathbb{N}} a \pi_i^{m+1}( b ) x^{i+n}. $
This proves the induction step.
Now we show the ''strong'' statement.
To this end, take $a \in R_x$ and $b \in R$.
Suppose first that (N5) holds.
Then $x(ab) = (xa)b$. Thus, since $a \in R_x$, we get that
$\delta(ab) + \sigma(ab)x = (ax)b 
\stackrel{(N3)}{=}
a(xb) =
a (\delta(b) + \sigma(b)x) = a\delta(b) + a(\sigma(b)x)
\stackrel{(N3)}{=}
a \delta(b) + (a \sigma(b))x$.
Hence by (N1), we get that $\delta(ab)=a\delta(b)$ and $\sigma(ab)=a\sigma(b)$.
Suppose now that (N4) holds.
Then $x(ba) = (xb)a$. Thus
$\delta(ba) + \sigma(ba)x = (\delta(b) + \sigma(b)x)a =
\delta(b)a + (\sigma(b)x)a 
\stackrel{(N3)}{=}
\delta(b)a + \sigma(b)(xa) = 
[a \in R_x] = \delta(b)a + \sigma(b)(ax) 
\stackrel{(N3)}{=}
\delta(b)a + (\sigma(b)a)x$.
Hence, by (N1), we get that $\delta(ba)=\delta(b)a$ and $\sigma(ba) = \sigma(b)a$.
Thus, in either case, $\sigma$ is a fixed point homomorphism of $R$
and $\delta$ is a $\sigma$-kernel derivation of $R$.
\end{proof}

For use in later sections, we now note
that the axioms (N4) and (N5) of Definition \ref{defstrong}
can be replaced by 
seemingly
stronger statements.

\begin{prop}\label{generalaxioms}
Let $(S,x)$ be a non-associative Ore extension of $R$.

\begin{itemize}

\item[(a)] The axiom {\rm (N4)} holds if and only if 
$(\mathbb{Z}[x] , S , R_x[x]) = \{ 0 \}$ holds.

\item[(b)] The axiom {\rm (N5)} holds if and only if 
$(\mathbb{Z}[x] , R_x[x] , S) = \{ 0 \}$ holds.

\end{itemize}
\end{prop}

\begin{proof}
Since the ''if'' statements are trivial,
we only show the ''only if'' statements.
To this end, take $a \in R_x$, $b \in R$
and $m,n,p \in \mathbb{N}$.

(a) We need to show that $(x^n , b x^m , ax^p) = 0$.
Since $x \in N_m(S) \cap N_r(S)$ and $a \in R_x$
it is enough to show this relation for $m=p=0$.
Since (N4) holds, we get, from the proof of Proposition \ref{necessarypolynomial}, that
$(x^n b)a = \sum_{i \in \mathbb{N}} \pi_i^n(b) x^i a =
\sum_{i \in \mathbb{N}} \pi_i^n(b) a x^i =
\sum_{i \in \mathbb{N}} \pi_i^n(ba) x^i = x^n(ba).$

(b) We need to show that $(x^n , a x^p , b x^m) = 0$.
Since $x \in N_m(S) \cap N_r(S)$ and $a \in R_x$
it is enough to show this relation for $m=p=0$.
Since (N5) holds, we get, from the proof of Proposition \ref{necessarypolynomial}, that
$(x^n a)b = (a x^n)b = a (x^n b) = 
\sum_{i \in \mathbb{N}} a \pi_i^n(b) x^i =
\sum_{i \in \mathbb{N}} \pi_i^n(ab) x^i = x^n(ab).$
\end{proof}

\begin{prop}\label{newproof}
If $S = R[X ; \sigma , \delta]$ is a generalized polynomial ring, then 
\begin{itemize}
\item[(a)] $R \subseteq N_l(S)$ if and only if $R$ is associative;
\item[(b)] $X \in N_l(S)$ if and only if
$\sigma$ is a ring endomorphism and
$\delta$ is a $\sigma$-derivation;
\item[(c)] $S$ is associative if and only if
$R$ is associative, $\sigma$ is a ring endomorphism and
$\delta$ is a $\sigma$-derivation.
\end{itemize}
\end{prop}

\begin{proof}
(a) The ''only if'' statement is clear.
Now we show the ''if'' statement.
Suppose that $R$ is associative.
Take $a,b,c \in R$ and $m,n \in \mathbb{N}$.
We wish to show that 
\begin{equation}\label{associativeabc}
(a , bX^m , cX^n)=0.
\end{equation}
Since $X \in N_r(S)$, we get that 
$(a , bX^m , cX^n) = (a , bX^m , c)X^n$.
Thus it is enough to prove \eqref{associativeabc} for $n=0$. 
Since $X \in N_m(S) \cap N_r(S)$ we get that
$(a , bX^m , c) = (a , b , X^m c) = 
\sum_{i \in \mathbb{N}} (a , b , \pi_i^m(c) X^i) = 
\sum_{i \in \mathbb{N}} (a , b , \pi_i^m(c) ) X^i = 0$,
using that $R$ is associative.

(b) First we show the ''only if'' statement.
Suppose that $X \in N_l(S)$. Take $a,b \in R$.
From the equality $X(ab) = (Xa)b$ we get that
$\delta(ab) + \sigma(ab)X = 
(\delta(a) + \sigma(a)X)b 
\stackrel{(N3)}{=}
\delta(a)b + \sigma(a) (Xb) = 
\delta(a)b + \sigma(a)( \delta(b) + \sigma(b)X ) 
\stackrel{(N3)}{=}
\delta(a)b + \sigma(a)\delta(b) + (\sigma(a)\sigma(b))X$.
Hence, by (N1), we get that 
$\sigma$ is a homomorphism and $\delta$ is a $\sigma$-derivation.
Now we show the ''if'' statement.
Suppose that $\sigma$ is a homomorphism and 
that
$\delta$ is a $\sigma$-derivation.
From the calculation in the proof of the ''only if''
statement it follows that $X \in N_l(R)$.
From the same type of reasoning that we used in the proof
of the ''if'' statement in (a), we therefore get that
$(X,S,S) \subseteq \sum_{i \in \mathbb{N}} (X,R,R)X^i = \{0\}$.

(c) The ''only if'' statement follows directly from (a) and (b).
Now we show the ''if'' statement.
Suppose that $R$ is associative, $\sigma$ is a ring endomorphism and
that
$\delta$ is a $\sigma$-derivation.
Take $a \in R$ and $m \in \mathbb{N}$.
From (a) and (b) we get that $a,X \in N_l(S)$.
Since $N_l(S)$ is multiplicatively closed we get that
$aX^m \in N_l(S)$. Since $N_l(S)$ is closed under addition,
we get that $S \subseteq N_l(S)$
and thus $S$ is associative.
\end{proof}

\begin{prop}\label{rightbasis}
If $S = R[X ; \sigma , \delta]$ is a generalized polynomial ring
with $\sigma$ bijective, then $B = \{ X^n \}_{n \in \mathbb{N}}$
is a basis for $S$ as a right $R$-module. 
\end{prop}

\begin{proof}
First we show that $B$ is a right $R$-linearly independent set.
We will show that for each $n \in \mathbb{N}$, the set 
$B_n := \{ X^i \}_{i=0}^n$ is right $R$-linearly independent. 
We will prove this by induction over $n$.
Base case: $n=0$. It is clear that $\{ 1 \}$ is right 
$R$-linearly independent.
Induction step: suppose that $B_n$ is right $R$-linearly independent
for some $n \in \mathbb{N}$.
Suppose that $a_i \in R$, for $i \in \{1 ,\ldots , n+1\}$, 
are chosen so that $\sum_{i=0}^{n+1} X^i a_i = 0$.
Then $0 = \sigma^{n+1}(a_{n+1}) X^{n+1} + \mbox{[lower terms]}$.
Since $B_{n+1}$ is left $R$-linearly independent, 
we get that $\sigma^{n+1}(a_{n+1}) = 0$.
Since $\sigma$ is injective, we get that $a_{n+1}=0$.
Thus $\sum_{i=0}^{n} X^i a_i = 0$.
By the induction hypothesis, we get that $a_i = 0$,
for $i \in \{0,\ldots,n\}$.

Next we show that $B$ right $R$-spans $S$.
For each $n \in \mathbb{N}$, let $S_n$ (or $T_n$) denote
the left (or right) $R$-span of $B_n$.
We will show that for each $n \in \mathbb{N}$, the 
relation $S_n = T_n$ holds.
We will prove this by induction over $n$.
Base case: $n=0$. It is clear that $S_0 = R = T_0$.
Induction step: suppose that $S_n = T_n$ for some
$n \in \mathbb{N}$.
Take $a = \sum_{i=0}^{n+1} a_i X^i \in S_{n+1}$.
Since $\sigma$ is surjective, we can pick $r \in R$
such that $\sigma^{n+1}(r) = a_{n+1}$.
This implies that $a - X^{n+1} r \in S_n$.
By the induction hypothesis this implies that 
$a - X^{n+1} r \in T_n$. Thus $a \in T_n + X^{n+1}r \subseteq T_{n+1}$.
Thus $S_{n+1} \subseteq T_{n+1}$.
Since the inclusion $S_{n+1} \supseteq T_{n+1}$ trivially holds,
the induction step is complete.
\end{proof}

Explicit formulas for how elements of generalized polynomial 
rings can be expressed as right $R$-linear combinations of 
elements from $B$
can be worked out exactly as in the classical case 
(see e.g. the formulas right after Theorem 7 in 
Ore's classical article \cite{ore1933}).
In this article, we only need the following special case of these relations.

\begin{prop}\label{rightformula}
Suppose that $S = R[X ; \id_R , \delta]$ is a non-associative 
differential polynomial ring.
If $r \in R$ and $n \in \mathbb{N}$, then
$r X^n = \sum_{i=0}^n (-1)^i {n \choose i} X^{n-i} \delta^i(r)$.
\end{prop}

\begin{proof}
We will show this by induction over $n$.
Base case: $n=0$. This is clear since $r X^0 = r = X^0 r$.
Induction step: suppose that 
$r X^n = \sum_{i=0}^n (-1)^i {n \choose i} X^{n-i} \delta^i(r)$
for some $n \in \mathbb{N}$.
Then, since $X \in N_m(S) \cap N_r(S)$, we get that 
$r X^{n+1} = r X^n X = 
\sum_{i=0}^n (-1)^i {n \choose i} X^{n-i} \delta^i(r) X =
\sum_{i=0}^n (-1)^i {n \choose i} X^{n-i} ( X \delta^i(r) - \delta^{i+1}(r) ) =  
\sum_{i=0}^n (-1)^i {n \choose i} X^{n+1-i} \delta^i(r) +
(-1)^{i+1} {n \choose i} X^{n-i} \delta^{i+1}(r) = 
[ {n+1 \choose i} = {n \choose i} + {n \choose i-1} ]=  
\sum_{i=0}^{n+1} (-1)^i {n+1 \choose i} X^{n+1-i} \delta^i(r)$
\end{proof}

\section{Ideal Structure}\label{simplicity}

The aim of this section is to prove Theorem \ref{maintheorem}.
To this end, we first show
a series of results concerning
simplicity and the center. 
Throughout this section, $R$ denotes a non-associative ring
and $\sigma$ and $\delta$ are additive maps $R \rightarrow R$
satisfying $\sigma(1)=1$ and $\delta(1)=0$. 
Furthermore, we let $S = R[X ; \sigma , \delta]$ denote a 
non-associative Ore extension of $R$.

\begin{defn}\label{defsigmadeltasimple}
An ideal $I$ of $R$ is said to be
\emph{$\sigma$-$\delta$-invariant}
if $\sigma(I) \subseteq I$ and $\delta(I) \subseteq I$.
If $\{ 0 \}$ and $R$ are the only $\sigma$-$\delta$-invariant ideals
of $R$, then 
$R$ is said to be \emph{$\sigma$-$\delta$-simple}.
\end{defn}

\begin{prop}\label{sigmadeltasimple}
If $S$ is simple, then $R$ is $\sigma$-$\delta$-simple.
\end{prop}

\begin{proof}
Take a non-zero $\sigma$-$\delta$-invariant ideal $J$ of $R$.
We wish to show that $J = R$.
Let $I = \oplus_{i \in \mathbb{N}} J X^i$.
Since $J$ is a right ideal of $R$ it follows that $I$ is a right ideal of $S$.
Using that $J$ is $\sigma$-$\delta$-invariant it follows that $I$ is a left ideal of $S$.
Since $J$ is non-zero it follows that $I$ is non-zero.
By simplicity of $S$, we get that $I=S$ and thus $J = R$.
\end{proof}

\begin{prop}\label{R-sigmadelta-simple-Z-field}
Suppose that $R$ is $\sigma$-$\delta$-simple.
If $\sigma$ is a fixed point homomorphism and $\delta$
is a $\sigma$-kernel derivation, then $Z(R)_{\delta}^{\sigma}$ is a field.
\end{prop}

\begin{proof}
Put $T = Z(R)_{\delta}^{\sigma}$.
We already know that $Z(R)$ is an associative commutative unital ring.
Suppose that $\sigma$ and $\delta$ are right $R_{\delta}^{\sigma}$-linear.
Take $a,b \in T$.
We have
$\sigma(ab) = \sigma(a)b = ab$ and
$\delta(ab) = \delta(a)b = 0b = 0$.
Thus $ab \in T$.
Since it is clear that $1 \in T$ and that
$T$ is additively closed, it follows that
$T$ is an associative commutative unital ring.
What remains to show is that every non-zero element of $T$ 
has a multiplicative inverse.
To this end, take a non-zero $a \in T$.
Then $Ra$ is a non-zero ideal of $R$ with
$\sigma(Ra) = \sigma(R)a \subseteq Ra$ and 
$\delta(Ra) = \delta(R)a \subseteq Ra$.
Hence $Ra$ is $\sigma$-$\delta$-invariant.
By $\sigma$-$\delta$-simplicity of $R$,
we get that $Ra = R$. Thus, there is $b \in R$
such that $ab = 1$.
By Proposition \ref{centerInvClosed}, we get that $b \in Z(R)$.
Now we show that $b \in R_{\delta}^{\sigma}$.
Indeed,
$\sigma(b)=\sigma(b)1=\sigma(b)ab=\sigma(ba)b=\sigma(1)b=1b=b$
and
$\delta(b) = \delta(b)1 = \delta(b)ab = \delta(ba)b =
\delta(1)b = 0b = 0$.
This shows that $b\in T$.
The left $R_{\delta}^{\sigma}$-linear case is treated analogously.
\end{proof}

\begin{prop}\label{commutativecondition}
If $a \in R_{\delta}^{\sigma}[X]$ commutes with 
every element of $R$, then $a \in C(S)$. 
\end{prop}

\begin{proof}
First we show, using induction, that, for
every $n \in \mathbb{N}$, the relation $[a,x^n]=0$ holds.
The base case $n=0$ follows immediately 
since $[a,X^0] = [a , 1] = 0$.
Now we show the induction step.
Suppose that $[a,X^n]=0$ for some $n \in \mathbb{N}$. Then 
$[a,X^{n+1}] = 
a X^{n+1} - X^{n+1} a =
a ( X X^n) - (X X^n) a =
(a X) X^n - X (X^n a) =
(a X) X^n - X (a X^n) =
(a X) X^n - (X a) X^n =
[a,X]X^n = 0$,
since it follows from $a \in R_{\delta}^{\sigma}[X]$ that $[a,X]=0$. Now 
$[a,bX^n] = 
a(bX^n) - (bX^n)a =
(ab)X^n - b(X^n a) =
(ab)X^n - b(a X^n) =
(ab)X^n - (ba)X^n = 
[a,b]X^n = 0$.
\end{proof}

\begin{prop}\label{associativecondition}
Suppose that $S$ is a strong non-associative Ore extension of $R$.
If $a \in R_{\delta}^{\sigma}[X]$
commutes with every element of $R$, and 
associates with all elements of $R$, then $a \in Z(S)$.
\end{prop}

\begin{proof}
By Proposition~\ref{commutativecondition} we conclude that $a\in C(S)$.
Since $Z(S) = C(S) \cap N(S)$,
we need to show that $a \in N(S)$.
First we show that $a \in N_l(S)$.
Take $n,p \in \mathbb{N}$.
Since $(a,R,R) = \{ 0 \}$ and $X \in N_m(S) \cap N_r(S)$, we get that
$(a, RX^n , RX^p) = (a , RX^n , R)X^p = (a , R , X^n R)X^p \subseteq
\sum_{i \in \mathbb{N}} (a , R , \pi_i^n(R) X^i) X^p \subseteq
\sum_{i = 1}^n (a , R , R) X^{i+p} = \{ 0 \}.$
By Proposition \ref{intersection}, we are done
if we can show that $a \in N_m(S)$ or $a \in N_r(S)$.

Case 1: (N4) holds.
We show that $a \in N_r(S)$.
We wish to show that 
\begin{equation}\label{rightzero}
(bX^n , cX^p , a) = 0.
\end{equation}
Since $X \in N_m(S) \cap N_r(S)$ and $a \in C(S)$, we get that
$( (bX^n) (cX^p) ) a = ( ((bX^n) c) X^p) a = ( (bX^n)c ) (X^p a) =
( (bX^n ) c ) (a X^p) =  (((bX^n)c) a)X^p = (( b (X^n c) ) a) X^p$
and, by Proposition \ref{generalaxioms}(a), we get that
$bX^n ( (cX^p) a ) = bX^n ( c (X^p a) ) = bX^n ( c ( a X^p ) ) =
bX^n ( (ca) X^p ) = (bX^n (ca) ) X^p 
= ( b ( X^n (ca))) X^p = ( b ( (X^n c) a ) ) X^p.$
This shows \eqref{rightzero}.

Case 2: (N5) holds.
We show that $a \in N_m(S)$.
We wish to show that
\begin{equation}\label{middlezero}
(bX^n , a , cX^p) = 0.
\end{equation}
Since $X \in N_r(S)$, we only need to show \eqref{middlezero} for $p=0$.
Since $X \in N_m(S) \cap N_r(S)$, $a \in C(S)$
and $a$ associates with all elements of $R$, we get that
$((bX^n)a)c = (b (X^n a) )c = (b (a X^n))c = 
( (ba)X^n ) c = (ba)(X^n c) 
= \sum_{i \in \mathbb{N}} (ba) \pi_i^n(c) X^i =
\sum_{i \in \mathbb{N}} b (a \pi_i^n(c) ) X^i.$ 
On the other hand, since $a \in C(S)$, $X \in N_m(S) \cap N_r(S)$ 
and Proposition \ref{generalaxioms}(b) holds, we get that
$bX^n (a c) = b (X^n (ac)) = b ( (X^n a) c) = 
b ( (a X^n) c) = b ( a (X^n c) ) =
\sum_{i \in \mathbb{N}} b ( a \pi_i^n(c) ) X^i.$
This shows \eqref{middlezero}.
\end{proof}

\begin{cor}\label{corcenter}
If $\delta$ is a kernel derivation on $R$ and we put $D = R[X ; \id_R , \delta]$, then
$Z(D)$ is the set of all $a \in D$ such that
(i) $a$ commutes with $X$, and
(ii) $a$ commutes with all elements of $R$, and
(iii) $a$ associates with all elements of $R$.
\end{cor}

\begin{prop}\label{regular}
Let $\sigma$ be injective and suppose that $a,b \in S=R[X;\sigma, \delta]$ are elements such that $ab=ba=1$. If the leading coefficient of $a$ is a regular element of $R$, then $a,b \in R$. 
\end{prop}

\begin{proof}
Suppose that $b = \sum_{i=0}^m b_i X^i$,  where $b_m \neq 0$.
Comparing coefficients of $X^{n+m}$ in the relation
$ab = 1$ we get that $a_n \sigma^n(b_m) = 0$
if $m + n > 0$.
Since $a_n$ is regular, we therefore get that 
$\sigma^n(b_m)=0$ whenever $m+n > 0$.
By injectivity of $\sigma$, we get $b_m=0$ if $m > 0$.
Comparing coefficients of degree $n$ in the relation $ba = b_0 a = 1$ 
we get that $b_0 a_n = 0$ if $n > 0$. Since $b_0 = b \neq 0$
and $a_n$ is regular, we get that $n=0$.
Hence $m=n=0$ and $a,b \in R$. 
\end{proof}

\begin{prop}\label{sumdegrees}
If $a,b \in S$, then
$\deg(ab) \leq \deg(a) + \deg(b)$.
Moreover, if $b$ is monic or $a$ is monic and $\sigma$ is injective,
then equality holds.
\end{prop}

\begin{proof}
Suppose that $\deg(a)=m$ and $\deg(b)=n$.
Let $a_m$ and $b_n$ denote the leading coefficients
of $a$ and $b$ respectively.
Then $ab = a_m \sigma^m(b_n) X^{m+n} + [\mbox{lower terms}]$.
So $\deg(ab) \leq m+n = \deg(a) + \deg(b)$.
Equality holds if and only if $a_m \sigma^m(b_n) \neq 0$.
This holds in particular if $b_n=1$ or if $a_m=1$ and $\sigma$ is injective.
\end{proof}

Next we show that there in some cases is a Euclidean algorithm for $S$.

\begin{prop}\label{euclideanalgorithm}
If $a,b \in S$ where $b$ is monic, then $a = qb + r$ 
for suitable $q,r \in S$ such that either $r=0$ or $\deg(r) < \deg(b)$.
\end{prop}

\begin{proof}
We follow closely the proof in \cite[p. 94]{rowen1988} for the associative case.
Without loss of generality, we may assume that $a\neq 0$.
Suppose that $\deg(a)=m$ and $\deg(b)=n$.
Let $a_m$ denote the leading coefficient of $a$.
Case 1: $m < n$. Then we can put $q=0$ and $r=a$.
Case 2: $m \geq n$.
Put $c = a - (a_m X^{m-n}) b$. Then $\deg(c) < \deg(a)$.
By induction there are $q',r' \in S$
with $c = q'b + r'$ and $r' = 0$ or $\deg(r') < n$. 
This implies that
$a = 
(a_m X^{m-n}) b + c = 
(a_m X^{m-n}) b + q'b + r' =  
(a_m X^{m-n} + q')b + r'.$
So we can put $q = a_m X^{m-n} + q'$ and $r=r'$.
\end{proof}

\subsection*{Proof of Theorem \ref{maintheorem}}

\subsubsection*{Proof of {\rm (a)}} 
Let $I$ be an ideal of $D$.
Suppose that $m$ is the minimal degree of non-zero
elements of $I$.
Put 
$J = \{ r \in R \mid \exists r_0,r_1, \ldots, r_{m-1} \in R : 
rX^m + r_{m-1} X^{m-1} + \ldots + r_0 \in I \}.$
It is clear that $J$ is a ideal of $R$.
From the fact that $XI - IX \subseteq I$ it follows that
$J$ is $\delta$-invariant.
Since $R$ is $\delta$-simple and $J$ is non-zero, 
we can conclude that $J = R$.
In particular, $1 \in J$.
Therefore there is a monic $a \in I$ of degree $m$.

Now we show that $a \in Z(D)$.
To this end, we check (i), (ii) and (iii) of Corollary \ref{corcenter}.
Since $a \in D_{\delta}$, (i) holds.
Now we check (ii). Take $r \in R$.
Since $a$ is monic the leading coefficient
of $[a,r]$ is $[1,r] = 0$.
Thus $\deg([a,r]) < m$ which, since $[a,r] \in I$,
implies that $[a,r]=0$, by minimality of $m$.
Now we check (iii).
Take $r,s \in R$.
Since $a$ is monic and the leading coefficients of all the polynomials
$(a,r,s)$, $(r,a,s)$ and $(r,s,a)$ equal zero, all
of them have degree less that $m$.
By minimality of $m$ and the fact that all of these polynomials
belong to $I$, we get that they are zero.
Thus (iii) holds.

Next we show that $I = Da$. 
The inclusion $I \supseteq Da$ is clear.
Now we show the reversed inclusion.
Take a non-zero $c \in I$.
Since $\deg(c) \geq \deg(a)$, we can use 
Proposition \ref{euclideanalgorithm} to conclude that
$c = qa + r$, for some $q,r \in S$ with $\deg(r) < \deg(a)$.
But then $r = c - qa \in I$, which, by minimality of $m$,
implies that $r = 0$. Therefore $c = qa \in I$.
Hence $I \subseteq Da$.

Finally we show uniqueness of $a$.
Suppose that $d \in D$ is monic and $I = Dd$. 
From the relations $a \in Dd$
and $d \in Da$ we get, respectively from Proposition \ref{sumdegrees}, that
$\deg(a) \geq \deg(d)$ and $\deg(d) \geq \deg(a)$,
which together imply that $\deg(a) = \deg(d)$.
Since $a$ and $d$ are monic,
we get that $\deg(a-d) < m$, which, by $a-d \in I$ and 
minimality of $m$, implies that $a=d$.

\subsubsection*{Proof of {\rm (b)}} 
Case 1: $Z(D)$ only contains polynomials of degree zero.
Then $Z(D) \subseteq Z(R)_{\delta}$. But since $Z(R)_{\delta} \subseteq Z(D)$
we get that $Z(D) = Z(R)_{\delta}$ and we can choose $b=1$.

Case 2: $Z(D)$ contains polynomials of degree greater than zero.
Let $n$ denote the least degree of non-constant polynomials in $Z(D)$.
Take $b \in Z(D)$ such that $\deg(b)=n$.
Now we show that we may choose $b$ to be monic.
Since $I = Db$ is an ideal of $D$, by (a), 
we may choose a monic $f \in I \cap Z(D)_{\delta}$
such that $I = Df$. But then $b = cf$ for some $c \in D$.
Since $f$ is monic we get that $n = \deg(b) = \deg(c) + \deg(f)$
which implies that $\deg(f) \leq n$. By minimality of $n$
we get that $\deg(f) = n$ and we may choose $b$ to be the monic $f$.

Now take $g \in Z(D)$ of degree $m$.
We will show by induction over the degree of $g$ 
that $g \in Z_{\delta}(R)[b]$.
Base case: $m=0$, i.e. $g$ is constant.
Then $g \in R \cap Z(S) = Z_{\delta}(R) \subseteq Z_{\delta}(R)[b]$.
Induction step: suppose that $m > 0$ and that we have shown the claim
for all $m' < m$.
Since $b$ is monic, we can write 
$g = hb + k$ for some $h,k \in S$ with $\deg(k) < \deg(b)$.
Note that, since $b$ is monic, we get that $\deg(h) < \deg(g)$.
We claim that $h,k \in Z(D)$.
If we assume that the claim holds, then, by the induction 
hypothesis, we are done.
Now we show the claim.
To this end, we will check (i),(ii) and (iii) in Corollary \ref{corcenter}.
First we check (i).
Note that $0 = [X,g] = [X,h]b + [X,k]$.
Seeking a contradiction, suppose that $[X,h] \neq 0$.
Since $b$ is monic and $\deg([X,k]) \leq \deg(k)$,
we get the contradiction
$-\infty = \deg(0) = \deg([X,g])= \deg( [X,h]b + [X,k] ) \geq n$.
Therefore $[X,h]=0$ and hence $[X,k]=0$.
In other words $h,k \in R_{\delta}[X]$.
Now we show (ii). 
To this end, note that
$0 = [r,g] = [r,h]b + [r,k]$.
Seeking a contradiction, suppose that $[r,h] \neq 0$.
Since $b$ is monic and $\deg([r,k]) \leq \deg(k)$,
we get the contradiction
$-\infty = \deg(0) = \deg([r,g]) = \deg( [r,h]b + [r,k] ) \geq n$.
Therefore $[r,h]=0$ and hence $[r,k]=0$.
Finally, we show (iii). Take $r,s \in R$.
Let $\alpha(\cdot)$ denote either of the maps
$(\cdot,r,s)$, $(r,\cdot,s)$ or $(r,s,\cdot)$. Then
$0 = \alpha(g) = \alpha(h)b + \alpha(k)$.
Seeking a contradiction, suppose that $\alpha(h) \neq 0$.
Since $b$ is monic and $\deg(\alpha(k)) \leq \deg(k)$,
we get the contradiction
$-\infty = \deg(0) = \deg(\alpha(g)) = \deg( \alpha(h)b + \alpha(k) ) \geq n$.
Therefore $\alpha(h)=0$ and hence $\alpha(k)=0$.
This completes the induction step.

Now we show uniqueness of $b$ up to addition by an element from $Z(R)_{\delta}$.
Case 1: $Z(D)$ only contains polynomials of degree zero.
Then there is only one monic polynomial in $Z(D)$, namely $b=1$.

Case 2: $Z(D)$ contains polynomials of degree greater than zero
i.e. $n > 0$.
Suppose that there is another monic $b' \in R_{\delta}[X]$
such that $Z(D) = Z(R)_{\delta}[b']$.
Then there is a polynomial $p \in Z(R)_{\delta}[X]$
such that $b = p(b')$. Hence $n = \deg(b) = \deg(p(b)) \geq \deg(b')$.
By minimality of $n$, we get that $\deg(b')=n$.
But then $b-b'$ is a polynomial in $Z(D)$
of degree less than $n$, which, by minimality of $n$,
implies that $b - b' \in Z(R)_{\delta}$.

\subsubsection*{Proof of {\rm (c)}} 
First we show the ''only if'' statement.
Suppose that $D$ is simple.
By Proposition \ref{sigmadeltasimple}, we get that 
$R$ is $\delta$-simple.
By Proposition \ref{centerfield}, we get that $Z(D)$ is a field.

Next we show the ''if'' statement.
Suppose that $R$ is $\delta$-simple and that $Z(D)$ is a field.
Let $I$ be a non-zero ideal of $D$.
By (a) and Proposition \ref{regular}, this implies that 
the polynomial in $Z(D)$ corresponding to $I$ is $1$.
This implies that $I = D$.

By (b) and  Proposition \ref{regular}, 
the ring $Z(R)_{\delta}[b]$ is a field precisely when $b=1$.

\subsubsection*{Proofs of {\rm (d)} and {\rm (e)}}
By Proposition \ref{rightbasis}, we can write $b = \sum_{i=0}^n b_i X^i$,
where $b_i \in R$, for $i\in \{1,\ldots,n\}$, with $b_n = 1$.
Since $b \in Z(D)$, we get, in particular, that $Xb = bX$.
This implies that $\delta(b_i)=0$, for $i\in\{1,\ldots,n\}$.
Therefore $b = \sum_{i=0}^n X^i b_i$.
For every  $j \in \{ 1,\ldots,n \}$ define the polynomial
$c_j = \sum_{i=j}^n X^{i-j} {i \choose j} b_i$.
We claim that each $c_j \in Z(D)$.
If we assume that the claim holds, then, by minimality of $n$,
we get that $b_j = c_j \in Z(R)_{\delta}$ and that
${i \choose j} b_i = 0$ whenever $1 \leq j < i \leq n$.
In the case when the characteristic of $Z(R)_{\delta}$ is zero,
we therefore get that $b=1$ or $b = b_0 + X$.
The relation $br=rb$ now gives us that
$\delta = \delta_{b_0}$.
Now suppose that the characteristic of $Z(R)_{\delta}$ is a prime $p$.
Fix $i \in \{ 1,\ldots,n \}$ such that $b_i$ is non-zero.
Then ${i \choose j} = 0$ when $1 \leq j < i$.
By Lucas' Theorem (see e.g. \cite{fine1947}) 
this implies that $i$ must be a power of $p$.
Choose the smallest $q \in \mathbb{N}$ such that $p^q \leq n$.
For each $i \in \mathbb{N}$ put $c_i = b_{p^i}$.
Also put $c = b_0$.
Then $b = c + \sum_{i = 0}^q c_i \delta^{p^i}$.
The relation $br=rb$ now gives us that
$\delta_c + \sum_{i=0}^n c_i \delta^{p^i} = 0$.

Now we show the claim.
To this end, we will check 
conditions (i), (ii) and (iii) of Corollary \ref{corcenter}.
Since $\delta(b_i)=0$ we know that (i) holds.
Now we show (ii).
Take $r \in R$. 
First note that since $br = rb$, we can use Proposition \ref{rightformula}
to conclude that 
\begin{equation}\label{notethat}
b_v r = \sum_{i=v}^n (-1)^{i-v} {i \choose i-v} \delta^{i-v}(r) b_i
\end{equation}
for each $v \in \{ 0,\ldots,n \}$. Thus,
\begin{align*}
c_j r &= \sum_{i=j}^n r \left( X^{i-j} {i \choose j} b_i \right) \stackrel{[X \in N_m(D)]}{=}
\sum_{i=j}^n \left( r X^{i-j} \right) {i \choose j} b_i \\
&=
\sum_{i=j}^n \left( \sum_{k=0}^{i-j} X^{i-j-k} (-1)^k {i-j \choose k} 
\delta^k(r) \right) {i \choose j} b_i  \stackrel{[X \in N_l(D)]}{=} \\
\end{align*}
\begin{align*}
&= \sum_{i=j}^n \sum_{k=0}^{i-j} X^{i-j-k} (-1)^k {i-j \choose k} 
\delta^k(r) {i \choose j} b_i  \stackrel{[v = i-k]}{=} \\
&= \sum_{i=v}^n \sum_{v=j}^n X^{v-j} {i \choose j}{i-j \choose i-v} 
(-1)^{i-v} \delta^{i-v}(r) b_i \\
&= \sum_{i=v}^n \sum_{v=j}^n X^{v-j} {v \choose j}{i \choose i-v} 
(-1)^{i-v} \delta^{i-v}(r) b_i \stackrel{[{\rm Eq.} \ \eqref{notethat}]}{=}
\sum_{v=j}^n X^{v-j} {v \choose j} b_v r = c_j r.
\end{align*}

Finally, we show (iii).
Take $r,s \in R$.
From the relations $(r,s,b)=0$ and $(b,r,s)=0$
it follows that $(r,s,b_i)=(b_i,r,s)=0$.
Hence we get that
$(r,s,c_j)=(c_j,r,s)=0$.
Thus $c_j \in N_r(R) \cap N_l(R)$.
Since $c_j \in C(R)$, we now automatically get that
$(r,c_j,s) = (r c_j)s - r(c_j s) = 
(c_j r) s - r (s c_j) = c_j (rs) - (rs) c_j = 0$. 
Hence $c_j \in N_m(R)$.
\hfill $\qed$

\begin{rem}
Our proof of Theorem \ref{maintheorem}(d)(e) follows closely
the proof of Amitsur \cite[Theorems 3 and 4]{amitsur1957} 
from the associative situation. We also remark that Amitsur's proof
is much simpler in characteristic $p>0$ 
than the proofs given later by Jordan \cite[Theorem 4.1.6]{jordan1975}
in the $\delta$-simple situation, although, as we show,
Amitsur's original proof can be adapted to this situation.
\end{rem}

\section{Non-associative Weyl Algebras}\label{sectionweyl}

In this section, we show that there are lots of natural examples
of non-associative diffe\-rential polynomial rings.
To this end, we introduce non-associative versions of the first Weyl algebra
(see Definition \ref{definitionweyl}) and we show that they are often 
simple regardless of the characteristic (see Theorem \ref{theoremweyl}).
Throughout this section, $T$ denotes a non-associative ring
and $T[Y]$ denotes the polynomial ring over
the indeterminate $Y$. In other words $T[Y] = T[Y ; \id_R , 0]$
as a generalized polynomial ring.

\begin{defn}\label{definitionweyl}
If $\delta : T[Y] \rightarrow T[Y]$
is a $T$-linear map such that $\delta(1)=0$,
then the non-associative differential polynomial ring
$T[Y] [X ; \id_R , \delta]$ is called a
{\it non-associative Weyl algebra}. 
\end{defn}

\begin{rem}
A non-associative Weyl algebra is a generalization of the classical (associative) first Weyl algebra, hence the name.
Recall that the first Weyl algebra,
$A_1(\C)=\C\langle X,Y\rangle / (XY-YX-1)$
may be regarded as a differential polynomial ring
$\C[Y][X;\identity_\C,\delta]$,
where $\delta : \C[Y] \to \C[Y]$
is the standard derivation on $\C[Y]$.
\end{rem}

\begin{thm}\label{theoremweyl}
If $T$ is simple and there for each positive $n \in \mathbb{N}$ is a non-zero $k_n \in Z(T)$
such that $\delta(Y^n) = k_n Y^{n-1}$, then
the non-associative Weyl algebra $T[Y] [X ; \id_R , \delta]$ is simple.
\end{thm}

\begin{proof}
Put $R = T[Y]$ and $S = R[X ; \id_R , \delta]$.
First we show that $R$ is $\delta$-simple.
Let $I$ be a non-zero $\delta$-invariant ideal of $R$.
Take a non-zero $a \in I$. Suppose that the degree of $a$ is $n$.
From the definition of $\delta$ it follows that $\delta^n(a)$
is a non-zero element of $I$ of degree zero.
This means that $I \cap T$ is non-zero.
By simplicity of $T$ it follows that $I \cap T = T$.
In particular $1 \in T = I \cap T \subseteq I$.
Hence $I = R$.

It is clear that $\delta$ is a kernel derivation.
Therefore, by Theorem \ref{maintheorem}(b)(c) we are done if we can show that
every non-zero monic $b \in R_{\delta}[X] \cap Z(S)$ is of degree zero.
It is clear that $R_{\delta} = T$.
Therefore $b \in T[X] \cap Z(S)$.
Seeking a contradiction, suppose that the degree of $b$ is $n > 0$.
Put $b = X^n + c X^{n-1} + [\mbox{lower terms}]$.
From $b \in Z(S)$ it follows that $c \in Z(T)$.
Take $r \in R$. Then
$0 = br - rb = (\delta(r) + cr - r c)X^{n-1} + [\mbox{lower terms}] = [c \in Z(T)] = 
\delta(r) X^{n-1} + [\mbox{lower terms}]$.
Thus $\delta(r) = 0$, for all $r \in R$, which is a contradiction
since e.g.
$\delta(Y) = k_1 \neq 0$.
\end{proof}

\begin{cor}
If $T$ is simple and $\delta$ is the classical derivative on $T[Y]$, then
the non-associative Weyl algebra $T[Y] [X ; \id_R , \delta]$ is simple
if and only if ${\rm char}(T)=0$.
\end{cor}

\begin{proof}
The ''if'' statement follows immediately from Theorem \ref{theoremweyl}
where $k_n = n$, for $n > 0$.

Now we show the ''only if'' statement. 
Suppose that ${\rm char}(T)=p>0$. 
Then $Y^p \in Z(T[Y] [X ; \id_R , \delta])$.
In particular, from Proposition \ref{regular}, we get that
$Z(T[Y] [X ; \id_R , \delta])$ is not a field.
By Theorem \ref{maintheorem}(c) we get that 
$T[Y] [X ; \id_R , \delta]$ is not simple.
As an alternative proof it is easy to see that
the proper non-zero ideal in 
$T[Y]$ generated by $Y^p$
is $\delta$-invariant.
Thus $T[Y]$ is not $\delta$-simple.
By Theorem \ref{maintheorem}(c), we get that 
$T[Y] [X ; \id_R , \delta]$ is not simple.
\end{proof}

\section{Kernel Derivations Defined by Automorphisms}\label{sectiondynamics}

In this section, we show 
simplicity results for a differential 
polynomial ring version of the quantum plane
(see Theorem \ref{theoremquantumtorus}) and
for differential polynomial rings defined by actions
on compact Hausdorff spaces (see Theorem \ref{theoremdynamics}).
To this end, we introduce a class of $\sigma$-kernel derivations
defined by ring morphisms
(see Definition \ref{definitionkernel}).
Throughout this section, $R$ denotes a non-associative ring.

\begin{prop}\label{definitionkernel}
If $\alpha : R \rightarrow R$ is a ring morphism,
then the map $\delta_{\alpha} : R \rightarrow R$
defined by $\delta_{\alpha}(r) = \alpha(r) - r$, for $r \in R$,
is a left and right $R_{\delta_{\alpha}}^{\identity_R}$-linear $\alpha$-kernel derivation.
Moreover, an ideal $I$ of $R$ is $\delta_{\alpha}$-simple
if and only if it is $\alpha$-simple.
\end{prop}

\begin{proof}
It follows immediately that $\delta_{\alpha}(1)=0$
and that $\delta_{\alpha}$ is additive.
Now we will show that $\delta_{\alpha}$ in fact is $R_{\delta_{\alpha}}^{\identity_R}$-linear
both from the left and the right.
In particular, $\delta_{\alpha}$ is an $\alpha$-kernel derivation.
Take $r \in R$ and $s \in \ker(\delta_\alpha)$.
Then $\delta_{\alpha}(rs) = 
\alpha(rs) - rs = 
\alpha(r)\alpha(s) - rs =
\alpha(r)s - rs = 
(\alpha(r) - r)s = 
\delta_{\alpha}(r)s$.
In the same way we get that $\delta_{\alpha}(sr) = s \delta_{\alpha}(r)$.
The last statement is clear since if $a \in I$,
then $\delta_\alpha(a) \in I$ if and only if $\alpha(a) - a \in I$.
\end{proof}

\begin{rem}\label{remarkderivation}
The $\alpha$-kernel derivation $\delta_{\alpha}$ from Proposition \ref{definitionkernel}
is seldom a derivation. In fact, suppose that $\delta_{\alpha}$ is a derivation.
Take $r,s \in R$. Then the relation 
$\delta_{\alpha}(rs) = \delta_{\alpha}(r)s + r \delta_{\alpha}(s)$
may be rewritten as
$\delta_{\alpha}(r) \delta_{\alpha}(s) = 0$.
So in particular, we get that $\delta_{\alpha}(r)^2 = 0$.
Hence, if $R$ is a reduced ring, i.e. a ring 
with no non-zero
nilpotent elements, then
$\delta_{\alpha}$ is a derivation if and only if $\alpha = \id_R$.
Thus, $\delta_\alpha$ would have to be the zero map.
\end{rem}

Let $T$ be a simple non-associative ring and suppose that $q \in Z(T) \setminus \{ 0 \}$.
Let $T[Y]$ denote the polynomial ring in the indeterminate $Y$ over $T$.
Define a ring automorphism $\alpha_q : T[Y] \rightarrow T[Y]$
by the $T$-algebra extension of the relation $\alpha_q(Y) = qY$.
By Proposition \ref{definitionkernel}, $\alpha_q$ in turn
defines an $\alpha$-kernel derivation $\delta_{\alpha_q} : T[Y] \rightarrow T[Y]$. 
It is not hard to show, using Remark \ref{remarkderivation},
that $\delta_{\alpha_q}$ is a classical derivation if and only if 
$q$ is nilpotent.

\begin{prop}\label{proprootunity}
If $T$ is simple, then $T[Y]$ is 
$\delta_q$-simple if and only if
$q$ is not a root of unity.  
\end{prop}

\begin{proof}
Put $R = T[Y]$.
First we show the ''only if'' statement.
Suppose that $q$ is a root of unity.
Take a non-zero $n \in \mathbb{N}$ with $q^n = 1$.
Then the ideal of $R$ generated by $Y^n$ is $\alpha_q$-simple.
Thus, $R$ is not $\alpha_q$-simple.
By Proposition \ref{definitionkernel} we get that $R$ is not $\delta_{\alpha_q}$-simple.
Now we show the ''if'' statement.
Suppose that $q$ is not a root of unity.
Take a non-zero $\delta_{\alpha_q}$-invariant ideal $I$ of $R$.
We wish to show that $I = R$.
By Proposition \ref{definitionkernel} $I$ is $\alpha_q$-invariant.
Take a non-zero $a \in I$ of least degree $m$.
Seeking a contradiction, suppose that $m > 0$.
Write $a = \sum_{i=0}^m a_i Y^i$, for some 
$a_i \in T$, for $i \in \{0,\ldots,n\}$. 
Then $\alpha_q(a) - k^m a$ is a non-zero element of $I$
of degree less than $m$. This contradicts the minimality of $m$.
Thus $m=0$ and thus $a \in I \cap T$.
Since $T$ is simple, we get that the ideal $J$ of $T$
generated by $a$ equals $T$. In particular, we get that
$I \supseteq T \ni 1$. Thus $I = R$.
\end{proof}

\begin{thm}\label{theoremquantumtorus}
If $T$ is simple and ${\rm char}(R)=0$, 
then the non-associative differential polynomial ring  
$D = T[Y][X ; \id_{T[Y]} , \delta_{\sigma_q}]$ is simple if and only 
if $q$ is not a root of unity.
In that case, $Z(D) = Z(T)$.
\end{thm}

\begin{proof}
The ''only if'' statement follows from Theorem \ref{maintheorem}(c)
and Proposition \ref{proprootunity}.
Now we show the ''if'' statement.
Put $R = T[Y]$ and $\delta = \delta_{\alpha_q}$.
Suppose that $q$ is not a root of unity.
By Proposition \ref{proprootunity},
we get that $R$ is $\delta$-simple. 
By Theorem \ref{maintheorem}(c), we are done if we can show that $Z(S)$ is a field.
To this end, we first note that, by Theorem \ref{maintheorem}(b), 
there is a unique monic $b \in Z(D)$ of least degree $n$.
Seeking a contradiction, suppose that $n > 0$.
Then $b = \sum_{i=0}^n b_i X^i$, 
for some 
$b_i \in R_{\delta}$.
But since $q$ is not a root of unity,
it follows that $R_{\delta} = T$.
Thus $b \in T[X]$.
From the fact that $b \in Z(D)$, we get that
$bt = tb$, for $t \in T$, which in turn implies that
$b_i \in Z(T)$, for $i\in\{0,\ldots,n\}$.
By looking at the degree $n-1$ coefficient in the relations
$br = rb$, for $r \in R$, we get that $\alpha_q = \id_R$,
which contradicts the fact that $q \neq 1$.
Thus $n=0$ and it follows that $b=1$.
By Theorem \ref{maintheorem}(b), we get that 
$Z(D) = Z(R)_{\delta}[1] = Z(T)$.
\end{proof}

\begin{rem}
Given a field $\mathbb{F}$ and $q\in \mathbb{F} \setminus\{0\}$, we may define the so called \emph{quantum plane} (see e.g. \cite[Chapter IV]{Kassel}) as $\mathbb{F}_q[X,Y] = \mathbb{F}\langle X,Y\rangle/(YX-qXY)$.
The quantum plane is an associative algebra and it can be realized as a classical Ore extension.
Indeed, if we define $\sigma : \mathbb{F}[X] \to \mathbb{F}[X]$ by $\sigma(X)=qX$, then the quantum plane $\mathbb{F}_q[X,Y]$
is isomorphic to the Ore extension $\mathbb{F}[X][Y,\sigma,0]$.
While the quantum plane can be seen as a $q$-deformation,
the non-associative Ore extension $D = T[Y][X ; \id_{T[Y]} , \delta_{\sigma_q}]$
that we study in Theorem \ref{theoremquantumtorus}
can be seen as a non-associative deformation of the plane.
\end{rem}

There are several ways to associate an (associative) algebra to
a dynamical system $(G,X)$, where $G$ is a group acting on a topological space $X$.
By associating a skew group algebra (see \cite{oinert2014}) or a crossed product $C^*$-algebra (see \cite{Power}) to the dynamical system,
it is possible to encode the dynamical system into the algebra
in such a way that
dynamical features (faithfulness, freeness, minimality etc) of the dynamical system
correspond to algebraical properties of the algebra.
We shall now show how to
associate a
non-associative differential polynomial ring
to a dynamical system
and exhibit a correspondence between minimality of the dynamical system
and simplicity of the non-associative ring.

For the rest of this section, let $K$ denote any of the real algebras 
$\mathbb{R}$ (real numbers), $\mathbb{C}$ (complex numbers), 
$\mathbb{H}$ (Hamilton's quaternions), 
$\mathbb{O}$ (Graves' octonions), 
$\mathbb{S}$ (sedenions), etc. 
obtained by iterating the classical Cayley-Dickson doubling 
procedure of the real numbers (for more details concerning this
construction, see e.g. \cite{baez2002}).
It is well known that $K$ is then a reduced ring.
Also, apart from the cases when $K$ equals $\mathbb{R}$,
$\mathbb{C}$ or $\mathbb{H}$, $K$ is not associative.
Furthermore, there is an $\mathbb{R}$-linear involution 
$\overline{\cdot} : K \rightarrow K$ and a norm
$| \cdot | : K \rightarrow \mathbb{R}_{\geq 0}$
satisfying $k \overline{k} = |k|^2$, for $k \in K$.
For the rest of this section, let $Y$ be a compact Hausdorff space
and let $g : Y \to Y$ be a continuous map.
A closed subspace $Z$ of $Y$ is called $g$-invariant
if $g(Z) \subseteq Z$.
The action of $g$ on $Y$ is called {\it minimal}
if $\emptyset$ and $Y$ are the only $g$-invariant subspaces of $Y$.
By abuse of notation, we let $C(Y)$ denote the ring of continuous functions $Y \rightarrow K$.
Since $K$ is reduced, we get that $C(Y)$ is also reduced.
The homeomorphism $g : Y \rightarrow Y$ defines a ring homomorphism
$\sigma(g) : C(Y) \rightarrow C(Y)$, where
$\sigma(g)(f) = f \circ g$, for $f \in C(Y)$.
By Proposition \ref{definitionkernel}, $\sigma(g)$ in turn
defines a $\sigma$-kernel derivation $\delta_{\sigma(g)} : C(Y) \rightarrow C(Y)$. 
Note that, by Remark \ref{remarkderivation}, $\delta_{\sigma(g)}$
is a classical derivation if and only if $g = \id_Y$.

\begin{prop}\label{NyPropMinimal}
If the action of $g$ on $Y$ is minimal,
then the ring $C(Y)$ is $\delta_{\sigma(g)}$-simple.
\end{prop}

\begin{proof}
Suppose that $Y$ is $g$-minimal.
We show that $C(Y)$ is $\delta_{\sigma(g)}$-simple.
Suppose that $I$ is a non-zero $\delta_{\sigma(g)}$-invariant ideal of $C(Y)$.
For a subset $J$ of $I$ define $N_J = \cap_{f \in J} f^{-1}(0)$.
Since $I$ is $\sigma(g)$-invariant it follows that $N_I$ is $g$-invariant.

It is clear that $N_J$ is closed.
Since $I$ is non-zero it follows that $N_I$ is a proper subset of $Y$.
By $g$-minimality of $Y$, we get that $N_I$ is empty.
By compactness of $X$ we get that there is some finite subset $J$ of $I$
such that $N_J$ is empty.
Define $h \in I$ by $h = \sum_{f \in J} f \overline{f} = \sum_{f \in J} |f|^2$.
Since $N_J$ is empty we get that $h(x) \neq 0$ for all $x \in X$.
Therefore $I$ contains the invertible element $h$ and hence
$I = C(Y)$.
\end{proof}

\begin{thm}\label{NYtheoremdynamics}
The non-associative differential polynomial ring 
$D = C(Y)[X ; \id_{C(Y)} , \delta_{\sigma(g)}]$ 
is simple if the action of $g$ on $Y$ is minimal
and the topology on $Y$ is non-discrete.
In that case, $Z(D) = \mathbb{C}$, if $K = \mathbb{C}$, and
$Z(D) = \mathbb{R}$, otherwise.
\end{thm}

\begin{proof}
Put $R = C(Y)$.
Suppose that the action of $g$ on $Y$ is minimal.
By Proposition \ref{NyPropMinimal} we get that $R$ is $\delta_{\sigma(g)}$-simple.
By Theorem \ref{maintheorem}(c) we are done if we can show that $Z(D)$ is a field.
To this end, we first note that, by Theorem \ref{maintheorem}(b), 
there is a unique monic $b \in Z(D)$ (up to addition of elements of $Z(R)_\delta$, which are of degree $0$) of least degree $n$.
Seeking a contradiction, suppose that $n > 0$.
Then $b = \sum_{i=0}^n b_i X^i$, 
for some
$b_i \in R_{\delta_{\sigma(g)}}$.
Take $i \in \{ 0,\ldots,n \}$
and $k_i \in b_i(Y)$.
Since
$b_i \in R_{\delta_{\sigma(g)}}$,
we get that
the set $b_i^{-1}(k_i)$ is non-empty and $g$-invariant.
By $g$-minimality of $Y$, we get that $Y = b_i^{-1}(k_i)$,
i.e. $b_i$ is the constant function $k_i$.
Thus $b = \sum_{i=0}^n k_i X^i$.
From the fact that $b \in Z(D)$, we get that
$bk = kb$, for $k \in K$, which in turn implies that
$b_i \in Z(K)$, for $i\in \{0,\ldots,n\}$.
By looking at the degree $n-1$ coefficient in the relations
$br = rb$, for $r \in R$, we get that $\sigma(g) = \id_R$.
Since the topology on $Y$ is non-discrete,
this contradicts $g$-minimality of $Y$.
Thus $n=0$ and it follows that $b=1$.
Thus, by Theorem \ref{maintheorem}(b), we get that 
$Z(D) = Z(R)_{\delta_{\sigma(g)}}[1] = Z(K)$.
It is well known that $Z(K)=\mathbb{R}$ for all $K$
except $K = \mathbb{C}$.
\end{proof}

\begin{prop}\label{propminimal}
Suppose that $g : Y \to Y$ is a homeomorphism.
The ring $C(Y)$ is $\delta_{\sigma(g)}$-simple
if and only if the action of $g$ on $Y$ is minimal.
\end{prop}

\begin{proof}
The ''if'' statement follows from Proposition \ref{NyPropMinimal}.

Now we show the ''only if'' statement.
Suppose that $C(Y)$ is $\delta_{\sigma(g)}$-simple.
We show that $Y$ is $g$-minimal.
Suppose that $Z$ is a closed $g$-invariant subset of $Y$
with $Z \subsetneq Y$.
We wish to show that $Z = \emptyset$.
To this end, let $I_Z$ denote the set of continuous functions
$X \rightarrow \mathbb{C}$ that vanish outside $Z$.
It is clear that $I_Z$ is an ideal of $C(Y)$.
It is also clear that $I_Z \subsetneq C(Y)$
since all non-zero constant maps belong to $C(Y) \setminus I_Z$.
Now we show that $I_Z$ is $\delta_{\sigma(g)}$-invariant.
Take $f \in I_Z$ and $x \in Y \setminus Z$.
Then $\delta_{\sigma(g)}(f)(x) = \sigma(g)(f)(x) - f(x) = [f \in I_Y \Rightarrow f(x)=0] =
\sigma(g)(f)(x) = f(g(x))=0$. The last equality follows since $g(x) \in Y\setminus Z$.
Now we prove this. Seeking a contradiction, suppose that $g(x) \in Z$.
Then, by the $g$-invariance of $Z$, we get $g^{-1}(Z)=Z$,
and $x = g^{-1}(g(x)) \in Z$, which is a contradiction.
By $\delta_{\sigma(g)}$-simplicity of $C(Y)$ this implies that
$I_Z = \{ 0 \}$. Since $Y$ is compact, it is completely regular.
Therefore, we get that $Z = \emptyset$.
\end{proof}

\begin{thm}\label{theoremdynamics}
Suppose that $g : Y \to Y$ is a homeomorphism.
The non-associative differential polynomial ring 
$D = C(Y)[X ; \id_{C(Y)} , \delta_{\sigma(g)}]$ 
is simple if and only if the action of $g$ on $Y$ is minimal
and the topology on $Y$ is non-discrete.
In that case, $Z(D) = \mathbb{C}$, if $K = \mathbb{C}$, and
$Z(D) = \mathbb{R}$, otherwise.
\end{thm}

\begin{proof}
Put $R = C(Y)$.
The ''if'' statement follows from Theorem \ref{NYtheoremdynamics}.

Now we show the ''only if'' statement.
Suppose that $S$ is simple.
By Theorem \ref{maintheorem} and Proposition \ref{propminimal}.
it follows that the action of $g$ on $Y$ is minimal.
Seeking a contradiction, suppose that the topology on $Y$ is discrete.
Since the topology is Hausdorff it follows that $Y$ is a one-element set.
Thus $S$ equals the polynomial ring $K[X]$ which is not simple.
\end{proof}

\section{Associative Coefficients}\label{sectionassociative}

In this section, 
we show that if the ring of coefficients is associative, 
then we can often obtain simplicity of the differential polynomial ring
just from the assumption that the map $\delta$ is not a derivation.

\begin{thm}\label{theoremassociative}
Suppose that $D = R[X ; \id_R , \delta]$ is a non-associative
differential polynomial ring such that $R$ is associative
and all positive integers are regular in $R$.
If $R$ is $\delta$-simple but $\delta$ is not a derivation,
then $D$ is simple.
\end{thm}

\begin{proof}
Let $I$ be a non-zero ideal of $D$. 
We wish to show that $I=D$. 
Pick a non-zero element $b \in I$ of least degree $n$.
Let $b=\sum_{i=0}^n c_i X^i$, for some $c_0,\ldots,c_n \in R$.
By mimicking the proof of Theorem \ref{maintheorem}(a),
we can conclude that we may choose $c_n=1$.
Seeking a contradiction, suppose that $n > 0$. 
We claim that $(b,d,e)=0$, for all $d,e \in R$. 
If we assume that the claim holds, then
by extracting the terms of degree $n-1$ from the
relation $(b,d,e)=0$ we get that
$n d\delta(e) +n\delta(d)e+ (c_{n-1}d)e-n\delta(de)-c_{n-1}(de) = 0$.
But since $R$ is associative and $n$ is regular,
this implies that 
$d \delta(e) + \delta(d)e = \delta(de)$
which contradicts the fact that $\delta$ is not a derivation.
Thus $n=0$ and hence $1 = b \in I$ which in turn implies that $I=D$.
Now we show the claim.
The degree $n$ part of $(b,d,e)$ equals
$(c_n d)e - c_n (de) = (1 \cdot d)e - 1 \cdot (de) = 0$.
Thus, since $(b,d,e) \in I$, we get that $(b,d,e)=0$,
from the minimality of $n$.
\end{proof}

\begin{rem}
In the cases when $T$ is associative, i.e.
in the cases when $K=\mathbb{R}$, $K = \mathbb{C}$ or $K = \mathbb{H}$,
then Theorem \ref{theoremassociative} can be used to simplify the proofs
of Theorem \ref{theoremquantumtorus} and Theorem \ref{theoremdynamics}.
\end{rem}

\end{document}